\documentclass{amsart}
\usepackage{amsmath}
\usepackage{amssymb}
\usepackage[all]{xy}
\usepackage{graphicx}
\usepackage{mathrsfs}

\numberwithin{equation}{section}

\usepackage{url}

\usepackage[latin1]{inputenc}
\usepackage{xspace,amssymb,amsfonts,euscript}
\usepackage{amsthm,amsmath}
\usepackage{palatino}
\usepackage{euscript}
\input xy \xyoption {all}

\usepackage{tikz}

\RequirePackage{color}
\definecolor{myred}{rgb}{0.75,0,0}
\definecolor{mygreen}{rgb}{0,0.5,0}
\definecolor{myblue}{rgb}{0,0,0.65}

\RequirePackage{ifpdf}
\ifpdf
  \IfFileExists{pdfsync.sty}{\RequirePackage{pdfsync}}{}
  \RequirePackage[pdftex,
   colorlinks = true,
   urlcolor = myblue, 
   citecolor = mygreen, 
   linkcolor = myred, 
   pagebackref,
   bookmarksopen=true]{hyperref}
\else
  \RequirePackage[hypertex]{hyperref}
\fi

\usepackage{cleveref}

\RequirePackage{ae, aecompl, aeguill} 

    \def\AM{{\mathbb{A}}}
    
    \def\CM{{\mathbb{C}}}

    \def\FM{{\mathbb{F}}}

    \def\QM{{\mathbb{Q}}}

    \def\ZM{{\mathbb{Z}}}

    \def\HC{{\mathcal{H}}}

    \def\LC{{\mathcal{L}}}

    \def\RC{{\mathcal{R}}}

\def\ES{{\EuScript E}}
\def\FS{{\EuScript F}}

\def\LS{{\EuScript L}}

\def\a{\alpha}

\def\l{\lambda}

\newcommand{\nc}{\newcommand} \newcommand{\renc}{\renewcommand}

\newcommand{\rdots}{\mathinner{ \mkern1mu\raise1pt\hbox{.}
    \mkern2mu\raise4pt\hbox{.}
    \mkern2mu\raise7pt\vbox{\kern7pt\hbox{.}}\mkern1mu}}

\def\un{\underline}

\def\to{\rightarrow}

\def\longto{\longrightarrow}

\nc{\triright}{\stackrel{[1]}{\to}}
\nc{\longtriright}{\stackrel{[1]}{\longto}}

\nc{\Hb}{H^\bullet}

\nc{\Br}{\mathcal{B}}
\nc{\HotRR}{{}_R\mathcal{K}_R}
\nc{\HotR}{\mathcal{K}_R}
\nc{\excise}[1]{}
\nc{\defect}{\text{df}}
\nc{\h}[1]{\underline{H}_{#1}}

\nc{\Ga}{\mathbb{G}_a} 
\nc{\Gm}{\mathbb{G}_m} 

\nc{\Perv}{{\mathbf{P}}}

\nc{\IH}{{\mathrm{IH}}}

\nc{\ic}{\mathbf{IC}}

\nc{\gl}{{\mathfrak{gl}}}
\renc{\sl}{{\mathfrak{sl}}}
\renc{\sp}{{\mathfrak{sp}}}

\renc{\Im}{\textrm{Im}}

\nc{\HBM}{H^{BM}}

 \DeclareMathOperator{\ch}{ch}

\DeclareMathOperator{\rank}{rank}

\newtheorem{thm}{Theorem}[section]
\newtheorem{lem}[thm]{Lemma}

\newtheorem{prop}[thm]{Proposition}
\newtheorem{cor}[thm]{Corollary}

\theoremstyle{definition}

\newtheorem{Eg}[thm]{Example}
\newtheorem{ex}[thm]{Exercise}

\theoremstyle{remark}
\newtheorem{remark}[thm]{Remark}
\newtheorem{warning}[thm]{Warning}
\newtheorem{question}[thm]{Question}

\nc{\simto}{\stackrel{\sim}{\to}}

\title{A reducible characteristic variety in type $A$}

\author{Geordie Williamson}
\address{Max-Planck-Institut f\"ur Mathematik, Vivatsgasse 7, 53111 Bonn, Germany}
\email{geordie@mpim-bonn.mpg.de}
\urladdr{http://people.mpim-bonn.mpg.de/geordie/}

\begin{document}

\begin{abstract}
We show that simple highest weight modules for
$\mathfrak{sl}_{12}(\CM)$ may have reducible characteristic
variety. This answers a question of Borho-Brylinski and Joseph from 1984.
The relevant singularity under Beilinson-Bernstein localization is the (in)famous Kashiwara-Saito
singularity. We sketch the rather indirect route via the
$p$-canonical basis, $W$-graphs and decomposition numbers for perverse sheaves
that led us to examine this singularity. 
\end{abstract}

\maketitle

\begin{center}
 \emph{Dedicated to David Vogan on the occasion of his $60^{th}$ birthday.}
\end{center}

\section{Introduction}

Let $G \supset B \supset T$ denote respectively a complex reductive
group, a Borel subgroup and maximal torus. Let $W$ denote its Weyl
group, $X = G/B$ the flag variety and $T^*X$ its cotangent
bundle. Given $x \in W$ we denote by $C_x$ the corresponding Schubert
cell and $T_x^* X \subset T^*X$ its conormal bundle. Let  $D_X$ denote
the sheaf of algebraic differential operators on $X$ and by $\LC_y$
the IC extension of the trivial local system on $C_y$. We can write
the characteristic cycle of $\LC_y$ as
\[
CC(\LC_y) = \sum_{x \in W} m_{x,y} [\overline{T_{x}^*X}].
\]
We have $m_{x,y} \in \ZM_{\ge 0}$ and $m_{x,y} = 0$ unless $x \le y$ in the
Bruhat order. The calculation of the multiplicities
$m_{x,y}$ is an important and difficult problem.
The question we address in this note is:

\begin{question} \label{q}
(See \cite[Conjecture 4.5]{BoBr} and \cite[\S 10.2]{Joseph})
Suppose that $G= SL_n(\CM)$.
Is $m_{x,y} = 0$ if $x \ne y$
  and $x$ and $y$ lie in the same two-sided Kazhdan-Lusztig cell? \end{question}

This question is equivalent to asking
whether the characteristic variety of a simple highest weight module
for $\sl_n(\CM)$ is irreducible \cite[Proposition 6.9]{BoBr}. (A sketch: if $\pi : T^*(G/B) \to
\sl_n(\CM)^*$ denotes the moment map then the characteristic variety of
the global sections of $\LC_y$ (a simple highest weight module) agrees with the image of the
characteristic variety of $\LC_y$ under $\pi$ \cite[Corollary 1.5]{BoBr}. The condition on two-sided cells occurs because if $x <_{LR} y$ ($\le_{LR}$ denotes the Kazhdan-Lusztig
two-sided cell preorder) then 
$\pi(\overline{T^*_xG/B})$ has strictly smaller dimension than $\pi( \overline{T^*_y G/B})$ and
hence cannot contribute a reducible component, because
characteristic varieties of simple modules are equidimensional \cite{Gabber}.)
It is known that reducible characteristic varieties occur in other
types (e.g. $B_2$, $B_3$, $C_3$) thanks to calculations of
Kashiwara and Tanisaki \cite{KT} and Tanisaki \cite{T}.

Kazhdan and Lusztig conjectured (still for $G = SL_n(\CM)$) that the
characteristic varieties of all $\LC_y$ are irreducible \cite{KLCC} (that is, that 
$m_{x,y} = 0$ if $x \ne y$). Of course this
would imply an affirmative answer to the above question. However
Kashiwara and Saito \cite{KS} showed that their conjecture was true if $n < 8$
but false for $n \ge 8$. They discovered a singularity
(the \emph{Kashiwara-Saito singularity}) which occurs as a
normal slice to a Schubert variety in the flag variety of $SL_8(\CM)$,
and for which the
characteristic variety is reducible. In their example $x$ and $y$ do not
lie in the same two-sided cell, and hence do not provide an example of
a reducible characteristic variety of a highest weight module.

In this note we give two permutations $x \le y$ in $S_{12}$ which
lie in the same right cell and such that a normal slice to the
Schubert variety corresponding to $y$ along the Schubert cell
corresponding to $x$ is isomorphic to the Kashiwara-Saito
singularity. This implies that $m_{x,y} \ne 0$, and hence that
Question \ref{q} has a negative answer.

\subsection{Structure of the paper}
In \S\ref{sec:motivation} we discuss the $p$-canonical basis and prove
a result relating characteristic cycle multiplicities and the $p$-canonical
basis. This result is a simple consequence of an observation of
Vilonen and the author
\cite{VW}. We then discuss how positivity properties of the
$p$-canonical basis and computer code of Howlett-Nguyen allows one to
narrow the search for potential counter-examples.
(Indeed, with
12! = 479 001 600 Schubert varieties in the flag variety of $GL_{12}$, the challenge is in the
finding rather than the verifying!)
In \S\ref{sec:calculation} we give the singularity in
the $GL_{12}$ flag variety and perform a straightforward
calculation to obtain the
Kashiwara-Saito singularity.

\subsection{Comments on the literature}

In \cite{Mel} a proof is proposed for the irreducibility of
characteristic varieties in type $A$. As we have already remarked,
this would imply that Question
\ref{q} has a positive answer. The results of this paper 
contradict \cite[Proposition 3.2]{Mel} and it is not clear to the
author how this proposition follows from the results of \cite{Joseph}. A statement
equivalent to \cite[Proposition 3.2]{Mel} is made in the remark
on page 54 of \cite{BoBr}.

\subsection{Acknowledgements}
This paper also owes a significant debt to Leticia Barchini who asked
me repeatedly about Question \ref{q}, and answered 
questions during and following a visit to the MPI last
year. Thanks also to Peter Trapa for some explanations and Anna
Melnikov, Yoshihisa
Saito and Toshiyuki Tanisaki for useful correspondence. The examples were found using Howlett and
Nguyen's software \cite{HN} for magma \cite{magma} which produces the irreducible W-graphs for the symmetric group,
implementing an algorithm described in \cite[\S 6]{HNPaper}.

During a visit to MIT last year 
David Vogan asked me whether the results of
\cite{VW} could produce new examples of reducible characteristic cycles, and
asked about Question \ref{q}. 
 It is a pleasure to dedicate this paper
to David, thank him for his many wonderful contributions to 
Lie theory and to wish him a happy birthday!

\section{Motivation from modular representation theory} \label{sec:motivation}

In this section we sketch the route which led us to consider the
singularity in $\S \ref{sec:GL12}$. We have tried to
provide enough details and references that a motivated reader could adapt these
techniques to find other interesting (counter)examples. Most of the
ideas are already contained in \cite{Wil}, which has more detail than
the discussion below.

\subsection{The $p$-canonical basis} \label{sec:pcan}

Let $G, B, T$ be as in the introduction. Let $W$ denote the Weyl
group, $S$ its simple reflections, $\le$ its Bruhat order and $\ell$ its length function.  Consider the flag variety
$G/B$ with its stratification by $B$-orbits (the Schubert stratification):
\[
G/B = \bigsqcup_{w \in W} C_w.
\]
Fix a field $\Bbbk$ of characteristic $p \ge 0$ and let $D^b_{(B)}(G/B; \Bbbk)$ denote the
bounded derived category of constructible sheaves on $G/B$ which are
constructible with respect to the Schubert stratification. For $w \in
W$ denote by $\ic(w;\Bbbk)$ the intersection cohomology sheaf
and $\ES(w;\Bbbk)$ the parity sheaf (for the constant
  pariversity) \cite{JMW2,Wil} corresponding to
$\overline{C_w}$. We will drop the $\Bbbk$ from the notation if it
is clear from the context.
If $\Bbbk$ is of characteristic 0 then $\ES(w; \Bbbk) = \ic(w; \Bbbk)$.

Let $\mathcal{H}$ denote the Hecke algebra of $(W,S)$. It is a
free $\ZM[v^{\pm 1}]$-module with basis $\{ H_w \; | \; w \in W \}$
and multiplication determined by
\[
H_sH_w = \begin{cases} H_{sw} & \text{if $\ell(sw) > \ell(w)$,} \\
(v^{-1} - v)H_w + H_{sw} & \text{if $\ell(sw) < \ell(w)$.} \end{cases}
\]
Let $\{ \un{H}_w \}$ denote the Kazhdan-Lusztig or ``canonical'' basis
of $\HC$.  We use the normalizations of \cite{SoeKL}. For example $\un{H}_s =
H_s + v H_{id}$.

Given a finite
dimensional $\ZM$-graded vector space $V = \bigoplus V^i$ let
\[
\mathrm{ch} V = \sum \dim_{i \in \ZM}
V^{-i}v^i \in \ZM[v^{\pm 1}]
\]
denote its Poincar\'e polynomial. Given $\FS \in D^b_{(B)}(G/B;\Bbbk)$
define
\[
\ch \FS = \sum_{x \in W} \ch H^*(\FS_x) v^{-\ell(x)} H_x
\in \HC
\]
where $\FS_x$ denotes the stalk of $\FS$ at the point $xB/B \in C_x
\subset G/B$. It is a classical theorem of Kazhdan and Lusztig \cite{KL2}
(see also \cite{SpIH}) that if $\Bbbk$ is of characteristic zero then
\begin{equation} \label{eq:KL}
\ch \ic(w;\Bbbk) = \un{H}_w.
\end{equation}
For any $w \in W$ we define
\[
{}^p\underline{H}_w := \ch \ES(w;\Bbbk).
\]
(One can show that ${}^p\underline{H}_w$ only depends on the
characteristic $p$ of $\Bbbk$, which explains the notation.) We call the $\{ {}^p\underline{H}_w \}$ the
\emph{$p$-canonical basis} for reasons which the following proposition
should make clear:

\begin{prop} ~ \label{prop:pcan}
  \begin{enumerate}
  \item[i)] ${}^p\underline{H}_w = H_w + \sum_{x < w} {}^ph_{x,w} H_x$ with
  ${}^ph_{x,w}  \in \ZM_{\ge 0}[v^{\pm 1}]$ (hence $\{ {}^p\underline{H}_w \; | \; w \in
  W\}$ is a basis),
\item[ii)] ${}^p\underline{H}_w = \sum {}^pm_{x,w} \underline{H}_x$
  for self-dual ${}^pm_{x,w} \in
  \ZM_{\ge 0}[v^{\pm 1}]$,
\item[iii)] if ${}^p m_{x,w}$ are as in (ii) then ${}^pm_{x,w} = 0$
  unless $\LC(x) \supset \LC(w)$ and $\RC(x) \supset \RC(w)$
  where $\LC$ and $\RC$ denote left and right descent sets,
\item[iv)] ${}^p\underline{H}_x{}^p\underline{H}_y = \sum
  {}^p\mu_{xy}^z {}^p\underline{H}_z$ for self-dual ${}^p\mu_{x,y}^z \in
  \ZM_{\ge 0}[v^{\pm 1}]$,
\item[v)] for $p \gg 0$, ${}^p\un{H}_w = {}^0 \un{H}_w = \un{H}_w$.
  \end{enumerate}
\end{prop}

\begin{proof}[Sketch of proof]
By definition the parity sheaf $\ES(w)$ is supported on
$\overline{C_w}$ and its restriction to  $C_w$ is isomorphic to a
shifted constant sheaf. (i) now follows easily from the definition of
$\ch$.

Each $\ES(w;\FM_p)$ admits a lift 
$\ES(w;\ZM_p)$, a parity sheaf with coefficients in
$\ZM_p$. Then $\ES(w, \ZM_p) \otimes_{\ZM_p} \QM_p$ is a parity sheaf
with coefficients in $\QM_p$, and is hence isomorphic to a direct sum
of intersection cohomology complexes. (ii) now follows from
\eqref{eq:KL} and the fact that $\ES(w;\FM_p)$, $\ES(w,\ZM_p)$ and $\ES(w,
  \ZM_p) \otimes_{\ZM_p} \QM_p$ all have the same character (see
  \cite[Theorem 3.10]{Wil}).

For fixed $w$ the parity sheaf $\ES(w; \FM_p)$ may be obtained via
pull-back from the partial flag variety $G/P$ where $P \supset B$ is the
parabolic subgroup determined by $\RC(w) \subset S$
(see \cite[Proposition 4.10]{JMW2}). Hence $\RC(x) \supset \RC(w)$ as
claimed. The statement for left descent sets follows because
${}^pm_{x,w} = {}^pm_{x^{-1},w^{-1}}$ by  \cite[\S 3 eq. (4)]{Wil}.

Each parity sheaf admits a lift to the $B$-equivariant derived category $D^b_B(G/B,\Bbbk)$ where there is a
convolution formalism categorifying the multiplication in the Hecke
algebra. (iv) then follows because the convolution of two
parity sheaves is isomorphic to a direct sum of shifts of parity
sheaves \cite[Theorem 4.8]{JMW2}.

Finally (v) follows from \eqref{eq:KL}  and \cite[Proposition 2.41]{JMW2} 
which asserts that $\ES(w;\FM_p) = \ic(w;\FM_p)$ for all but finitely many primes $p$.
\end{proof}

\begin{warning} The $p$-canonical basis depends on the root system of
  $G$, not just on its Weyl group. (For example the 2-canonical basis
  differs in types $B_3$ and $C_3$.) Hence one should think about the
  $p$-canonical basis as a basis of the Hecke algebra attached to a 
  root system or Cartan matrix rather than a Coxeter system. Kashiwara and Saito have
  observed the same phenomenon for characteristic cycles \cite[Example 5.4]
  {KT}.
\end{warning}

\subsection{The $p$-canonical basis and decomposition numbers}

We briefly recall the notion of a decomposition
number for perverse sheaves. An excellent reference is
\cite{decperv}.

Let $X$ denote a complex variety, $Z \subset X$ a locally closed
smooth subvariety, and $\LS$ a local system of free $\ZM$-modules on
$X$. One may consider the intersection cohomology
extension\footnote{For the perversity $p$, see \cite{decperv}.}
$\ic(\overline{Z};\LS)$. It is a perverse sheaf with $\ZM$-coefficients
on $X$. One has
\[
\ic(\overline{Z}; \LS)  \otimes_\ZM \QM = \ic(\overline{Z}; \LS \otimes_\ZM \QM)
\]
and so $\ic(\overline{Z}; \LS)$ can be thought of as a $\ZM$-form
of $\ic(\overline{Z}; \LS \otimes \QM)$. In general,
\[
\ic(\overline{Z}; \LC)  \otimes^L_\ZM \FM_p
\in D^b_c(X; \FM_p)
\]
is perverse but no longer simple. The decomposition matrix 
encodes the Jordan-H\"older multiplicities of the simple perverse
sheaves occurring in $\ic(\overline{Z}; \LC)  \otimes^L_\ZM
\FM_p$.

In this paper we will be concerned with the flag variety together with
its Schubert stratification, as in \S\ref{sec:pcan}. In this case all
the strata are simply connected and the decomposition matrix takes the
form $(d_{y,x})_{y,x \in W}$ where
\[
d_{y,x} := [ \ic(\overline{C_y};\ZM) \otimes_{\ZM} \FM_p:
\ic(\overline{C_x};\FM_p)].
\]

The relation between the characters of the parity sheaves (i.e. the
$p$-canonical basis) and the decomposition matrix is 
subtle. For example, recent papers of Achar and Riche \cite{AR1,AR2} prove that knowledge
of the $p$-canonical basis gives (a $q$-refinement of) the
decomposition matrix for perverse sheaves on the
Langlands dual flag variety.

Here we will be concerned with a much more limited but simpler
relationship. Roughly it says that the first time the $p$-canonical
basis differs from the canonical basis corresponds to the first
non-trivial decomposition number (see Proposition \ref{prop:pcan} and above for notation):

\begin{prop}
  Fix $y \in W$ and suppose that $x < y$ is maximal in the Bruhat
  order such that ${}^p
  m_{x,y} \ne 0$. If ${}^p m_{x,y} \in \ZM$ then $d_{y,x} = {}^p m_{x,y}$.
\end{prop}

\begin{proof}
  Fix $x$ and $y$ are in the proposition. Set
  \begin{gather*}
    X = \bigsqcup_{z \ge x} BzB/B, \qquad 
Z = BxB \subset X, \qquad
U = X \setminus U
  \end{gather*}
and denote by $i$ (resp. $j$) the closed (resp. open) embedding of $Z$
(resp. $U$) into $X$. Note that $X$ is open in $G/B$.

For $\Bbbk \in \{ \FM_p, \ZM_p, \QM_p \}$ let
$\ic_\Bbbk$ (resp. $\ES_{\Bbbk}$) denote the intersection cohomology
(resp. parity) sheaf corresponding to the stratum $ByB/B \subset
X$. We have $\ES_{\QM_p} \cong \ic_{\QM_p}$ and our assumptions
guarantee that $\ES$ is perverse with
\[
\ic_{\Bbbk|U} \cong \ES_{\Bbbk|U}.
\]
Hence we need to examine the difference between $\ic_{\FM_p}$ and
$\ES_{\FM_p}$ over the closed stratum $Z$.

Our main tool will be \cite[Lemma 2.18]{JMW2} which gives a
bijection between isomorphism classes of extensions of a fixed 
$\FS$ on $U$ to $X$, and
isomorphism classes of distinguished triangles on $Z$ of the form
\begin{equation}
  \label{eq:ext}
  A \to i^*j_*\FS \to B \triright.
\end{equation}
If $\FS'$ is such an extension then $A$ and $B$ are given by
\begin{equation}
  \label{eq:stalkscostalks}
  i^*\FS' \cong A \quad \text{and} \quad i^!\FS' \cong B[-1].
\end{equation}

Let us examine the triangle corresponding to the
extension $\ES_{\ZM_p}$ of $\ic_{\ZM_p|U}$. It has the form
\begin{equation}
  \label{eq:parityAB}
A \to i^*j_* (\ic_{\ZM_p|U}) \to B \triright  
\end{equation}
Because $Z$ is contractible we can view \eqref{eq:parityAB} as a distinguished
triangle of $\ZM_p$-modules. By \eqref{eq:stalkscostalks}  and the fact that $\ES$ is a parity
sheaf we deduce:
\begin{enumerate}
\item[(1)] $H^m(A)$ and $H^m(B)$ are free $\ZM_p$-modules;
\item[(2)] $H^m(A) = 0$ if $m - \ell(y)$ is odd, and $H^m(B) = 0$ if $m -
  \ell(y)$ is even.
\end{enumerate}
The assumptions of the proposition and  \eqref{eq:stalkscostalks} guarantee that
\begin{enumerate}
\item[(3)] $H^m(A)$ vanishes for $m > -\ell(x)$ and $H^m(B)$ vanishes for $m<-\ell(x)-1$;
\item[(4)] $H^{-\ell(x)}(A)$ is free of rank ${}^pm_{x,y}$ (in
  particular $\ell(y) - \ell(x)$ is even).
\end{enumerate}
Because $\ZM_p$ is hereditary, each of the terms in
\eqref{eq:parityAB} is isomorphic to its cohomology. Hence we can
turn the triangle and rewrite it as
\[
H^*(B)[-1] \to  H^*(A) \to H^*( i^*j_*
\FS ) \triright.
\]
By (3) above the only non-zero map component of the first map is
\[
\alpha : H^{-\ell(x)-1}(B) \to H^{-\ell(x)}(A).
\]
Because $\ES$ is indecomposable, $\alpha$ does not
map any summand of $H^{-\ell(x)-1}(B)$ isomorphically onto a summand
of $H^{-\ell(x)}(A)$ by \cite[Lemma 2.21]{JMW2}. In other words,
$\alpha \otimes_{\ZM_p} \FM_p = 0$. On the other hand,
we have
\[
\ES \otimes_{\ZM_p} \QM_p \cong \ic_{\QM_p} \oplus \ic(Z)^{\oplus ({}^pm_{x,y})}
\]
and hence $\a$ is an isomorphism over $\QM_p$. In other words, $\ker
\a = 0$ and the domain and codomain of $\a$ are free of the same rank.

By the long exact sequence of cohomology we deduce that:
\begin{gather*}
H^m(i^*j_* \ic_{\FM_p|U}) = \begin{cases}
H^m(A) \otimes \FM_p & \text{if $m < \ell(x) - 1$,} \\
H^{-\ell(x)}(A) \otimes \FM_p & \text{if $m = -\ell(x) - 1$ or
  $m=-\ell(x)$},\\
H^{m}(B) \otimes \FM_p & \text{if $m > -\ell(x)$}.
\end{cases} \\
H^m(i^*j_* \ic_{\QM_p|U}) = \begin{cases}
H^m(A) \otimes \FM_p & \text{if $m < \ell(x) - 1$,} \\
0 & \text{if $m = -\ell(x) - 1$ or
  $m=-\ell(x)$},\\
H^{m}(B) \otimes \FM_p & \text{if $m > -\ell(x)$}.
\end{cases}
\end{gather*}
By the Deligne construction \cite[Proposition 2.1.11]{BBD} we have
\[
i^* \ic_{\Bbbk} = i^*\tau_{< -\ell(x)} j_* \ic_{\Bbbk} = \tau_{<
  -\ell(x)}  i^*  j_* \ic_{\Bbbk}.
\]
where $\tau_{< m}$ denotes truncation.
Hence if $\chi$ denotes the Euler characteristic at any point in $Z$
we have
\[
\chi(\ic_{\FM_p})= \chi(\ic_{\QM_p}) - (-1)^{-\ell(x) } ({}^p m_{x,y}).
\]
Now we are done: if we write
\[
[\ic_{\ZM_p} \otimes^L_{\ZM_p} \FM_p] =  [\ic_{\FM_p}] + a [\ic(Z)]
\]
in the Grothendieck group of $\FM_p$-perverse sheaves on $X$ 
then taking Euler characteristics over $Z$ yields $a = {}^p m_{x,y}$ as claimed.
\end{proof}

As in the introduction we write $m_{x,y}$ for the the characteristic cycle
multiplicities.
The following is an immediate consequence of the previous proposition,
and \cite[Theorem 2.1]{VW}.

\begin{cor} \label{mpm}
  Suppose that $x < y$ are as in the previous proposition. Then
\[
m_{x,y} \ge {}^p m_{x,y}.
\]
\end{cor}

\subsection{Searching for a counter-example} Consider the following
variant of Question \ref{q} (with notation as in Proposition \ref{prop:pcan}):

\begin{question} \label{q2} Suppose that $G= SL_n(\CM)$ and let $p$
  be a prime. Is ${}^pm_{x,y} = 0$ if $x \ne y$
  and $x$ and $y$ lie in the same two-sided cell? \end{question}

It will become clear below that a positive answer to Question \ref{q} implies a
positive answer to Question \ref{q2}. Question \ref{q2} is also important for
modular representation theory, with connections to Lusztig's conjecture around the
the Steinberg weight \cite{Soe}, amongst other things.

One can show (using Soergel calculus \cite{EW} or 
Schubert calculus \cite{HW}) 
that the counter-example in \S\ref{sec:calculation} also
gives a counter-example to Question \ref{q2}. We found the examples by
pursuing a naive idea, which is the main theme of \cite{Wil}: the
$p$-canonical basis has remarkable positivity properties (summarized in
Proposition \ref{prop:pcan}) and these positivity properties are enough to
rule out many potential counter-examples.

Asume that $W$ is an arbitrary Weyl group.
For any left cell $C \subset W$ we can consider the corresponding cell
module
\[
M_C = \bigoplus_{x \in C} \ZM[v^{\pm 1}] M_x := \bigoplus_{x \le_L C}
\ZM[v^{\pm 1}] \un{H}_x / (\bigoplus_{x <_L C}\ZM[v^{\pm 1}]\un{H}_x ).
\]
The $\HC$-module structure in the basis $\{ M_x \}$ is encoded in the $W$-graph of $C$. Fix a
prime $p$ and assume that the $p$-canonical basis satisfies:
\begin{equation}
\label{cellcond}
\text{for all $y \in C$ if }{}^p m_{x,y} \ne 0 \text{ then } x \le_{L} y.
\end{equation}
Then we may define ${}^pM_y$ as the image of ${}^p \un{H}_y$ in $M_C$
and
obtain in this way a $p$-canonical basis for the cell module $M_C$. By
Proposition \ref{prop:pcan} it satisfies the following properties:
\begin{enumerate}
\item (positive upper-triangularity) we have 
\[
{}^pM_y = M_y + \sum_{C \ni x <
    y} {}^pm_{x,y} M_x \text{ with ${}^p m_{x,y} \in \ZM_{\ge
      0}[v^{\pm 1}]$ self-dual;} \]
\item (positive structure constants) for any $x \in W$,\[{}^p\un{H}_x \cdot {}^p M_y \in
  \bigoplus_{z \in C} \ZM_{\ge 0}[v^{\pm 1}] ({}^pM_z).\]
\end{enumerate}

\begin{Eg}
  Suppose that $W$ is of type $B_2$ with simple reflections $s,
  t$. Consider the left cell $C = \{ s, ts, sts \}$. The $W$-graph is:
\[
\xymatrix{
\{ s \} \ar@{-}[r] & \{t \} \ar@{-}[r] & \{ s \}
}
\]
In this case there are two possible bases for $M_C$ satisfying
(1) and (2). The first is the Kazhdan-Lusztig basis $\{ M_x \}$. The second is the basis $\{ M_x' \}$ with $M_x' = M_x$ for $x
  \in \{ s, ts \}$ and $M_{sts}' := M_{sts} + M_s$. In this case $M'$
  agrees with the image of the 2-canonical basis for $B_2$ (for an
  appropriate choice of long and short root).
\end{Eg}

Now assume that $W$ is of type $A_{n-1}$. In this case two sided cells are
parametrized by partitions $\lambda$ of $n$. Also, all left cells in
fixed two sided are irreducible and afford isomorphic (based)
representations of the Hecke algebra $\mathcal{H}$.

\begin{lem}
  Let $\l$ be a partition of $n$ and $E_\l \subset W$ the
  corresponding two-sided cell. Then there exists a left cell $C
  \subset E_\l$ satisfying \eqref{cellcond}.
\end{lem}

\begin{proof}[Sketch of proof] Let $w_\l$ denote the longest element of the standard parabolic
  subgroup $W_\l \subset W$ determined by $\l$. Then $w_\l \in
  E_\l$. We claim that the left cell $C$ containing $w_\l$ satisfies
  \eqref{cellcond}. Firstly, ${}^p\un{H}_{w_\l} = \un{H}_{w_\l}$ by
  (i) and (iii) of Proposition \ref{prop:pcan} and a simple induction
  then shows that
\[
{}^p \un{H}_y = \sum_{x \le_{L} w_\l} {}^pm_{x,y} \un{H}_x.
\]
for all $y \in C$. Hence \eqref{cellcond} holds.
\end{proof}

It follows that any left cell representation in type $A$ admits a
$p$-canonical basis satisfying the above positive conditions. One can
apply computer searches in order to isolate potential
counter-examples and then use Soergel calculus \cite{EW} or 
Schubert calculus \cite{HW} to check whether one has indeed found a
counter-example.

In order to implement this approach one needs the $W$-graphs of the
left cell representations in type $A$. These are provided by the
wonderful code of Howlett and Nguyen \cite{HN} for magma \cite{magma}.

\begin{remark}~
  \begin{enumerate}
  \item Using the recent results of Achar-Riche \cite{AR1,AR2} one can
    show that if there is a counter-example for a left cell
    corresponding to $\lambda$ then there
    is also a counter-example for the left cell corresponding to the
    transposed partition $\lambda^t$. This allows one to
    roughly halve the number of left cells which one needs to
    consider. Experimentally, the above positivity properties are 
    more restrictive in left cells corresponding to partitions ``near the top'' of the dominance
    order. (For example for $S_4$ there is only one solution for
    the left cell corresponding to the partition $(3,1)$, whereas
    there are two for the partition
    $(2,1,1)$.)
  \item Lusztig has given a beautiful description of the $J$-ring for a fixed
    two-sided cell in $S_n$ as a (based) matrix ring. Using this result one can show that if
    the $p$-canonical basis is trivial (i.e. equal to the image of the
    Kazhdan-Lusztig basis) in a fixed left cell then
    Question \ref{q2} has a positive answer for that two-sided cell.
  \item The above methods yielded another counter-example to
    Question \ref{q2}, this time in $GL_{13}$:
    \begin{gather*}
      \label{eq:3}
x =  12132156543765438798765ba98c,\\
 y = 121321546543765438798765aba9876cba98
    \end{gather*}
Here we write $x$ and $y$ as words in the simple transpositions $1, \dots 9, a,
b, c$ of $S_{13}$. Yoshihisa Saito has informed me that in this case
one also obtains the Kashiwara-Saito singularity as a normal slice.
\end{enumerate}
\end{remark}

\section{Two realisations of the Kashiwara-Saito singularity} \label{sec:calculation}

\subsection{Notation} Fix a positive integer $ n \ge 1$.

Let $S_n$ denote the symmetric group, which we regard
as permutations of the set $\{1, \dots, n \}$. We view $S_n$ as a
Coxeter group with Coxeter generators the simple transpositions $s_i =
(i,i+1)$ for $1 \le i \le n-1$. We write $\ell$ for the length
function on $S_n$ and $\le $ for the Bruhat order.

We will usually write
permutations in ``string notation'' i.e. we write $x = x_1x_2 \dots
x_n$ to mean that $x$ is the permutation in $S_n$ which sends $1 \mapsto x_1$, $2 \mapsto
x_2$ etc. To avoid confusion when using string notation we extend our
alphabet of digits $1, \dots, 9$ by the letters $a, b \dots$ with
$a = 10$, $b = 11 \dots$.

Let $G = GL_n(\CM)$ denote the general linear group of invertible
complex matrices. Given $x = x_1x_2 \dots x_n \in S_n$ we will denote
by $\dot{x}$ the corresponding permutation matrix. That is $\dot{x}(e_i)
= e_{x_i}$ if $e_1, e_2, \dots, e_n$
denotes the standard basis of $\CM^n$.

Let $B \subset G$ denote the Borel subgroup of upper triangular
matrices. Let $G/B$ denote the flag variety. Given $y \in S_n$ we
denote by
\[
C_y = B\dot{y}B/B \subset G/B
\]
its Schubert
cell and by $Z_y$ the corresponding Schubert variety:
\[
Z_y := \overline{B\dot{y}B/B} = \overline{C_y}.
\]

\subsection{Equations for slices to Schubert cells} We recall how to explicitly
write down equations for slices to Schubert cells in Schubert
varieties. Everything here can be checked reasonably easily by hand with the
(possible) exception of the fact that the equations 
\eqref{rankconds} are complete.

Let $U_-, U \subset GL_n(\CM)$ the
subgroups of unipotent lower and upper triangular matrices respectively. The natural map
$U_- \to G/B$ is an open immersion, giving a coordinate patch
isomorphic to $\AM^{n \choose 2}$ around
the base point $B \in G/B$. Hence for any $x \in S_n$ the natural map
$\pi : \dot{x}U_- \to G/B$ gives a coordinate patch around $\dot{x}B \in
G/B$. For a permutation $x = x_1 \dots x_n$ we have:
\[
\dot{x}U_- = \{ g = (g_{i,j}) \in GL_n(\CM)\; | \; g_{x_i,i} = 1 \text{ and }
g_{x_i,j}  = 0 \text{ for $j > i$} \}.
\]
For $y \in S_n$ the inverse image
$\pi^{-1}(Z_y) \subset \dot{x}U_-$ is given by the equations (see
\cite{Fulton}, \cite[\S 3.2]{WY1} and \cite[\S 2.2]{WY2}):
\begin{equation} \label{rankconds}
\rank( (g_{i,j})_{a \le i \le n, 1 \le j \le b} ) \le \rank(
(\dot{y}_{i,j})_{a \le i \le n, 1 \le j \le b} ) \quad \text{for all}
\quad 1 \le a, b \le
n.
\end{equation}

We have
\[
\pi^{-1}(C_x)   := \{ g \in \dot{x} U_- \; | \; g_{i,j} = 0 \text{ for } i > x_j \}.
\]
Hence if we set
\[
N_x = \{ g \in \dot{x}U_- \; | \; g_{i,j} = 0 \text{ for } i < x_j \}
\]
then $N_x$ is a normal slice to $C_x$ in $\dot{x}U_-$. Hence the singularity
of $Z_y$ along $C_x$ is given by $N_x \cap \pi^{-1}(Z_y)$ which is
given by intersecting the linear equations describing $N_x$ with the
equations \eqref{rankconds}.

\begin{ex} Perhaps an example will help decipher the notation. Consider $n = 4$ and let $x =
  2143$ and $y = 4231$. We have
\[
N_x = \left \{ \left ( \begin{matrix} 0 & 1 & 0 & 0 \\ 1 & 0 & 0 & 0 \\ a &
    b & 0 & 1 \\ c & d & 1 & 0 \end{matrix} \right ) \; \middle | \; a, b,
c, d \in \CM \right \}
\]
and the rank conditions \eqref{rankconds} reduce in this case to the
single equation $ad-bc = 0$.
\end{ex}

\subsection{The Kashiwara-Saito singularity}
Let $M_2(\CM)$ denote the space of $2 \times 2$-complex matrices with
coefficients in $\CM$. Consider the space $S$ of matrices $M_i \in M_2(\CM)$ for $i \in \ZM/4\ZM$
satisfying the two conditions:
\begin{gather}
  \label{eq:1}
\det M_i = 0 \text{ for $i \in \ZM/4\ZM$,} \\
M_iM_{i+1} = 0 \text{ for $i \in \ZM/4\ZM$}.
\end{gather}
Clearly $S$ is an affine variety. One can show that it is
irreducible of dimension 8. We call $S$ (or more precisely the
singularity of $S$ at $0 := (0,0,0,0) \in S$)
the \emph{Kashiwara-Saito} singularity. In \cite{KS} it is shown that
the conormal bundle to $0$ is a component of the characteristic cycle
of the intersection cohomology $D$-module on $S$. In particular, the
characteristic cycle is reducible.

\subsection{Realisation in $GL_8$}

Now let $n = 8$ and consider the permutations
\[
u := 21654387, \quad v := 62845173.
\]
Then $u$ is the maximal element in the standard parabolic subgroup
$\langle s_1, s_3, s_4, s_5, s_7 \rangle$ of length $\ell(u) = 8$. We
have $u \le v$ and $\ell(v) = 16$.

The following is stated without proof in \cite[\S 8.3]{KS}. We give the proof
here because it is a good warm-up for the calculation in $GL_{12}$
which we need to perform next.

\begin{prop} \label{prop:KS1}
  The singularity of $Z_v$ along $C_u$ is isomorphic to $S$.
\end{prop}

\begin{proof} If $J := \left ( \begin{matrix} 0 & 1 \\ 1 & 0 \end{matrix} \right
)$ we have (as a matrix of block $2 \times 2$-matrices):
\[
N_u := \left \{ 
\left ( \begin{matrix}
J & 0 & 0 & 0 \\
A_1 & 0 & J & 0 \\
A_2 & J & 0 & 0 \\
A_0 & A_3 & A_4 & J \end{matrix} \right )  \; \middle | \; A_i \in M_2(\CM) \right \}. \]
Now
\[
\tiny \dot{v} = \left ( \begin{array}{cc|cc|cc|cc}
0 & 0 & 0 & 0 & 0 & 1 & 0 & 0 \\
0 & 1 & 0 & 0 & 0 & 0 & 0 & 0 \\
\hline
0 & 0 & 0 & 0 & 0 & 0 & 0 & 1 \\
0 & 0 & 0 & 1 & 0 & 0 & 0 & 0 \\
\hline
0 & 0 & 0 & 0 & 1 & 0 & 0 & 0 \\
1 & 0 & 0 & 0 & 0 & 0 & 0 & 0 \\
\hline
0 & 0 & 0 & 0 & 0 & 0 & 1 & 0 \\
0 & 0 & 1 & 0 & 0 & 0 & 0 & 0
\end{array} \right )
\]
and after some checking one sees that the rank conditions
\eqref{rankconds} reduce to the following equations:
\begin{gather}
A_0 = 0, \\
\label{eq:rank23} \rank \left ( \begin{matrix} A_2 & J \\ 0 & A_3 \end{matrix} \right )
\le 2, \\
\label{eq:rank12} \rank \left ( \begin{matrix} A_1 \\ A_2 \end{matrix} \right ) \le 1, \\
\label{eq:rank34} \rank \left ( \begin{matrix} A_3 & A_4 \end{matrix} \right ) \le 1, \\
\label{eq:rank1234} \rank \left ( \begin{matrix} A_1 & 0 & J \\ A_2 & J & 0 \\ 0 & A_3 &
    A_4 \end{matrix} \right ) \le 4.
\end{gather}

Now
\[
\left ( \begin{matrix} A_2 & J \\ 0 & A_3 \end{matrix} \right ) 
\left ( \begin{matrix} I & 0 \\ -JA_2 & J \end{matrix} \right ) =
\left ( \begin{matrix} 0 & I \\ -A_3JA_2 & A_3J \end{matrix} \right ) 
\]
and so \eqref{eq:rank23} is equivalent to $A_3JA_2 = 0$. Similarly one
may show that together \eqref{eq:rank23} and \eqref{eq:rank1234} are equivalent
to the conditions
\begin{equation} \label{eq:firstreduction1}
A_3JA_2 = 0 \text{ and } A_4JA_1 = 0.
\end{equation}
If we let $K := \left ( \begin{matrix} 0 & -1 \\ 1 & 0 \end{matrix}
\right )$ then \eqref{eq:rank12} is equivalent to the conditions:
\[
A_2KA_1^t = 0 \text{ and } \det A_1 = \det A_2 = 0.
\]
Similarly, \eqref{eq:rank34} is equivalent to the conditions
\[
A_4^tKA_3 = 0 \text{ and } \det A_3 = \det A_4 = 0
\]
Hence if we set
\[
A_1' := A_1^tJ, \quad A_2' := A_2K, \quad A_3' := A_3J, \quad A'_4 := A_4^tK.
\]
Then the relations become
\begin{gather*}
\det A_i' = 0 \text{ for $i \in \{1, 2, 3, 4 \}$,} \\
A_2'A_1' = A_3'A_2' = A_4'A_3' = A_1'A_4' = 0.
\end{gather*}
This is clearly isomorphic to the Kashiwara-Saito singularity.
\end{proof}

\subsection{Realisation in $GL_{12}$} \label{sec:GL12}

We will see how to realise the Kashiwara-Saito
singularity as a normal slice in $Z_y$ to a Schubert cell $C_x$. This time $x$ and $y$ belong to the same right cell.

Now let $n = 12$. 
Consider the permutations in $S_{12}$:
\[
x = 438721a965cb, \quad y = 4387a2c691b5.
\]
(Remember that we use string notation and $a = 10$, $b = 11, c = 12$.)
The Robinson-Schensted $P$ and $Q$ symbols of $x$ are
\[
P(x) = \hspace{3pt}
\begin{matrix} 1 & 5 & 9 & b \\
2 & 6 & a & c \\
3 & 7 \\
4 & 8 \end{matrix} \quad \text{and} \quad
Q(x) = \hspace{3pt}
\begin{matrix} 1 & 3 & 7 & b \\
2 & 4 & 8 & c \\
5 & 9 \\
6 & a \end{matrix} \; \;.
\]
The $P$ and $Q$ symbols of $y$ are
\[
P(y) = \hspace{3pt}
\begin{matrix} 1 & 5 & 9 & b \\
2 & 6 & a & c \\
3 & 7 \\
4 & 8 \end{matrix} \quad \text{and} \quad
Q(y) = \hspace{3pt}
\begin{matrix} 1 & 3 & 5 & 7 \\
2 & 4 & 9 & b \\
6 & 8 \\
a & c \end{matrix} \; \;.
\]
In particular, we conclude that $x$ and $y$ are in the same two-sided
cell (even the same right cell).

Reduced expressions for $x$ and $y$ are given by:
\begin{align*}
x & =
s_bs_5s_6s_7s_8s_9s_5s_6s_7s_8s_7s_1s_2s_3s_4s_5s_1s_2s_3s_4s_3s_1\\
y & = 
s_5s_6s_7s_8s_9s_as_bs_as_1s_2s_3s_4s_5s_6s_7s_8s_9s_7s_8s_4s_5s_6s_7s_1s_2s_
3s_4s_5s_3s_1
\end{align*}
We have $x \le y$, $\ell(x) = 22$ and $\ell(y) = 30$.

Recall the Kashiwara-Saito singularity $S$ from the previous section.

\begin{prop}
  The singularity of $Z_y$ along $C_x$ is isomorphic to $S$.
\end{prop}

\begin{proof} The normal slice $N_x$ to $C_x$ inside the full flag variety is
  given by the space of matrices
\[
\left ( \begin{matrix} 
0 & 0 & J & 0 & 0 & 0 \\
J & 0 & 0 & 0 & 0 & 0 \\
B_1 & 0 & A_1 & 0 & J & 0 \\
B_2 & J & 0 & 0 & 0 & 0 \\
B_3 & B_5 & A_2 & J & 0 & 0 \\
B_4 & B_6 & B_7 & A_3 & A_4 &  J \end{matrix} \right )
\]
where $J := \left ( \begin{matrix} 0 & 1 \\ 1 & 0 \end{matrix} \right
)$ as above and the $A_i$, $B_i$ are in $M_2(\CM)$. Now, after some
checking one sees that the rank conditions \eqref{rankconds} give that  
 the intersection of $N_x$ and
$Z_y$ is cut out by the equations:
\begin{gather}
B_i = 0 \\
\rank \left (\begin{matrix} A_3 & A_4 \end{matrix} \right ) \le 1,\\
\rank \left ( \begin{matrix} 0 & A_1 \\ J & 0 \\ 0 & A_2 \end{matrix}
\right ) \le 3 \quad \Leftrightarrow \quad \rank \left ( \begin{matrix} A_1 \\
    A_2 \end{matrix} \right ) \le 1, \\
\rank \left ( \begin{matrix} A_2 & J \\ 0 & A_3 \end{matrix} \right )
\le 2, \\
\rank \left ( \begin{matrix} 0 & A_1 & 0 & J \\ J & 0 & 0 & 0 \\ 0 &
    A_2 & J & 0 \\ 0 & 0 & A_3 & A_4 \end{matrix} \right ) \le 6 \quad
\Leftrightarrow \quad 
\rank \left ( \begin{matrix} A_1 & 0 & J \\ 
    A_2 & J & 0 \\ 0 & A_3 & A_4 \end{matrix} \right ) \le 4.
\end{gather}
Looking at the proof of Proposition \ref{prop:KS1} it is now clear
that $N \cap Z_y \cong S$, the Kashiwara-Saito singularity.
\end{proof}

\def\cprime{$'$}

\end{document}